\documentclass[a4paper,12pt]{amsart}
\usepackage{amsmath, amssymb}
\usepackage[abbrev]{amsrefs}
\usepackage{color}
\usepackage{hyperref}
\hypersetup{linktocpage=true,colorlinks=true,linkcolor=blue,citecolor=blue}
\usepackage{comment}

\newtheorem{thm}{Theorem}[section]
\newtheorem{prop}[thm]{Proposition}
\newtheorem{lem}[thm]{Lemma}
\newtheorem{cor}[thm]{Corollary}

\theoremstyle{definition}
\newtheorem{dfn}[thm]{Definition}
\newtheorem{ex}[thm]{Example}

\theoremstyle{remark}
\newtheorem{rem}[thm]{Remark}
\newtheorem{qst}[thm]{Question}

\newcommand{\R}{\ensuremath{\mathbb{R}}}
\newcommand{\F}{\mathcal{F}}

\newcommand{\X}{\mathcal{X}}

\newcommand{\cN}{\mathcal{N}}
\newcommand{\cZ}{\mathcal{Z}}
\newcommand{\cP}{\mathcal{P}}

\DeclareMathOperator{\supp}{supp}
\DeclareMathOperator{\diam}{diam}
\newcommand{\midd}{\mathrel{} \middle| \mathrel{}}

\newcommand{\kf}{d_{\mathrm{KF}}}
\newcommand{\prok}{d_{\mathrm{P}}}
\newcommand{\gp}{d_{\mathrm{GP}}}
\newcommand{\Lip}{{\mathcal{L}}ip}
\newcommand{\conc}{d_{\mathrm{conc}}}
\newcommand{\haus}{d_{\mathrm{H}}}
\newcommand{\id}{\mathrm{id}}
\newcommand{\pr}{\mathrm{pr}}
\newcommand{\Cb}{C_{\mathrm{b}}}

\newcommand{\tF}{\tau_\mathrm{F}}
\newcommand{\tK}{\tau\mathrm{K}}

\DeclareMathOperator{\OD}{ObsDiam}
\DeclareMathOperator{\Ls}{Ls}
\DeclareMathOperator{\Li}{Li}

\DeclareMathOperator{\dis}{dis}
\DeclareMathOperator{\df}{def}

\newcommand{\ep}{\varepsilon}
\newcommand{\lm}{\lambda}
\newcommand{\Lm}{\Lambda}
\renewcommand{\phi}{\varphi}

\makeatletter
\@namedef{subjclassname@2020}{%
\textup{2020} Mathematics Subject Classification}
\makeatother

\title[Topological aspects of $\X$]{Topological aspects of the space of \\ metric measure spaces}

\author[D.~Kazukawa]{Daisuke Kazukawa}
\address{Faculty of Mathematics, Kyushu University, Fukuoka 819-0395, JAPAN}
\email{kazukawa@math.kyushu-u.ac.jp}

\author[H.~Nakajima]{Hiroki Nakajima}
\address{Mathematical Sciences Course, Ehime University, Matsuyama 790-8577, JAPAN}
\email{nakajima.hiroki.nz@ehime-u.ac.jp}

\author[T.~Shioya]{Takashi Shioya}
\address{Mathematical Institute, Tohoku University, Sendai 980-8578, JAPAN}
\email{shioya@math.tohoku.ac.jp}

\date{April 13, 2023}
\subjclass[2020]{Primary 53C23, Secondary 54B20, 54E35}
\keywords{metric measure space, box distance, concentration topology, pyramid, weak topology, hyperspace}
\thanks{This work was supported by JSPS KAKENHI Grant Number JP22K20338, JP22K13908, JP19K03459.}

\begin{document}

\begin{abstract}
Gromov introduced two distance functions, the box distance and the observable distance, on the space of isomorphism classes of metric measure spaces and developed the convergence theory of metric measure spaces.
We investigate several topological properties on the space equipped with these distance functions toward a deep understanding of convergence theory.
\end{abstract}

\maketitle

\section{Introduction}
The study of convergence of metric measure spaces is one of central topics in geometric analysis on metric measure spaces.
This study originates in that of Gromov-Hausdorff convergence/collapsing of Riemannian manifolds, which widely been developed and applied to solutions to many significant problems in geometry and topology.

Gromov introduced two fundamental concepts of distance functions, the {\it box distance function} $\square$ and the {\it observable distance function} $\conc$, on the set, say $\X$, of isomorphism classes of metric measure spaces and developed his distinctive theory in \cite{Grmv}*{Chapter 3.$\frac{1}{2}_+$}.
The box distance function is simpler and is close to a metrization of measured Gromov-Hausdorff convergence. Besides, this distance is equivalent to the Gromov-Prokhorov distance introduced by Greven-Pfaffelhuber-Winter \cite{GPW} (see \cite{Lohr}).
The topology and the convergence notion given by these distance functions are widely used as in \cites{GPW2, St}. On the other hand, the observable distance function induces a very characteristic topology, called the {\it concentration topology}, based on the concentration of measure phenomenon due to L\'evy \cite{Levy} and V.~Milman \cite{VMil} (see also \cite{Led}).
The concentration topology is effective to capture the high-dimensional aspects of spaces and admits the convergence of many sequences whose dimensions are unbounded. The study of the concentration topology has been growing in recent years.

We focus on the topological aspects of the space $\X$ with respect to these distance functions $\square$ and $\conc$. The concentration topology is coarser than the topology induced by the box distance, which is called the {\it box topology} simply in this paper.
As fundamental properties, it is known that the space $(\X, \square)$ is separable, complete, and non-compact and that $(\X, \conc)$ is also separable but is not complete.
However, other topological properties have not yet been studied.
In this paper, we investigate several topological properties toward a deep understanding of the convergence theory.

Moreover, Gromov also introduced a natural compactification, denoted by $\Pi$, of $\X$ with respect to the concentration topology at the same time, which is one of powerful tools to study the concentration topology.
The topology of this compactification is called the {\it weak topology} and each element of $\Pi$ is called a {\it pyramid} (this name comes from its definition).
The space $\Pi$ of pyramids is interested in itself because this contains many infinite-dimensional objects, for example, the (virtual) infinite-dimensional Gaussian space.
We also investigate the weak topology on $\Pi$. Not only $\Pi$ is compact, but it also has already known that $\Pi$ is metrizable by the third author \cite{MMG}.

\subsection*{Around compactness}
We study some properties of $\X$ around compactness. The space $\X$ is globally non-compact with respect to both topologies, but it is also locally non-compact.

\begin{thm}\label{no_cpt_nbd}
Any metric measure space has no compact neighborhood with respect to both the box and concentration topologies. In particular, $\X$ is not locally compact in either topology.
\end{thm}

For the box topology, Theorem \ref{no_cpt_nbd} implies the following fact as a corollary since $(\X, \square)$ is a Baire space through the Baire category theorem.

\begin{cor}\label{sigma_cpt}
For the box topology, $\X$ is not $\sigma$-compact.
\end{cor}

On the other hand, since $(\X, \conc)$ is incomplete, for the concentration topology, the above argument is not applied. Actually, we obtain the new fact that the concentration topology is non-Baire.

\begin{thm}\label{conc_baire}
For the concentration topology, $\X$ is not a Baire space. In particular, $\X$ is not completely metrizable.
\end{thm}

In \cite{BCZ}*{Question 9.1}, it is asked if the Gromov-Hausdorff space is homeomorphic to the space $l^2$. The answer of the analogous question to $(\X, \conc)$ is negative because $l^2$ is a Baire space.

\subsection*{Around connectivity}
We next study some properties of $\X$ and $\Pi$ around connectivity. The following theorem is very clear as a global property.

\begin{thm}\label{contra}
For both the box and concentration topologies, $\X$ is contractible. Moreover, $\Pi$ is contractible in the weak topology.
In particular, all of them are path connected and simply connected.
\end{thm}

This theorem is proved by constructing explicit deformation retractions.
On the other hand, local properties are not clear and difficult to prove.
For the box topology, we obtain a geodesic between two metric measure spaces with respect to the metric $\square$. This is one of the most important results in this paper.

\begin{thm}\label{box_geod}
For the box distance function, $\X$ is a geodesic space.
In particular, $\X$ is locally path connected in the box topology.
\end{thm}

The existence of geodesics is important and useful geometrically not only topologically.
Moreover, we prove that any two distinct spaces have uncountably many geodesics between them with respect to the box distance function (see Theorem \ref{thm:multigeod}).
Therefore any geodesic branches everywhere and the Alexandrov curvature of $\X$ is not bounded from below nor from above with respect to the box distance function.
For the concentration and weak topologies, it is difficult to obtain a geodesic at present, but it is possible to show the local path connectivity.

\begin{thm}\label{loc_path_conn}
Both $\X$ with the concentration topology and $\Pi$ with the weak topology are locally path connected.
\end{thm}

As a consequence of the above discussion, we also obtain the following characteristic corollary for the weak topology.

\begin{cor}\label{Peano}
For the weak topology, $\Pi$ is a Peano space. Namely, $\Pi$ is a continuous image of the unit interval.
\end{cor}

The topological properties of $\X$ and $\Pi$ are summarized in the table below (see {\sc Table} 1).

\begingroup
\renewcommand{\arraystretch}{1.2}
\begin{table}[ht]
\caption{}
\begin{tabular}{|c|c|c|c|}\hline
& $(\X, \square)$ & $(\X, \conc)$ & $(\Pi, \rho)$ \\\hline\hline
compact & No & No & Yes \\\hline
separable & Yes & Yes & Yes \\\hline
complete${}^\dagger$ & Yes & No & Yes \\\hline\hline
locally compact & No${}^*$ & No${}^*$ & Yes \\\hline
$\sigma$-compact & No${}^*$ & Unknown & Yes \\\hline
Baire space & Yes & No${}^*$ & Yes \\\hline
Polish & Yes & No${}^*$ & Yes \\\hline\hline
(globally) contractible & Yes${}^*$ & Yes${}^*$ & Yes${}^*$ \\\hline
locally path connected & Yes${}^*$ & Yes${}^*$ & Yes${}^*$ \\\hline
geodesic space${}^\dagger$ & Yes${}^*$ & Unknown & Unknown \\\hline
\end{tabular}
\begin{tabular}{l}
${}^*$ indicates our new results in this paper. \\
${}^\dagger$ indicates geometric properties with respect to the standard metrics.
\end{tabular}
\end{table}
\endgroup

\subsection*{Revisit the weak topology on $\Pi$}
As an application of our results, we give a reinterpretation of the weak topology using the theory of hyperspace. Here, a topological space consisting of (closed) subsets of a topological space $X$ is called a {\it hyperspace} over $X$. A pyramid in $\Pi$ is originally defined as a ($\square$-closed) subset of $\X$ satisfying certain conditions. Therefore, it is very natural to focus on the relation with the hyperspace.

Let $X$ be a Hausdorff space and let $\F(X)$ be the set of all closed subsets of $X$. The Kuratowski-Painlev\'e convergence and the Fell topology on $\F(X)$ is well-studied. The Kuratowski-Painlev\'e convergence is topological if and only if $X$ is locally compact. Here, a convergence is topological provided that there exists a topology achieving it. Moreover, in this case, the Fell topology achieves the Kuratowski-Painlev\'e convergence actually.

We now consider the space $\F(\X, \square)$. By Theorem \ref{no_cpt_nbd}, the underlying space $(\X, \square)$ is not locally compact, so that the Kuratowski-Painlev\'e convergence is not topological in this case. However, the finest topology whose convergence is weaker than the Kuratowski-Painlev\'e convergence always exists. This topology, write $\tK$, is called the {\it topologization} of the Kuratowski-Painlev\'e convergence. It follows from the general theory of hyperspace that the hyperspace $(\F(\X, \square), \tK)$ is compact, $T_1$, and sequential (in particular, it is sequentially compact). Note that $\tK$ is strictly finer than the Fell topology. We have the new interpretation that the space $\Pi$ is a subspace of the compact hyperspace $(\F(\X, \square), \tK)$.

\begin{thm}\label{emb}
The inclusion map $\Pi \ni \cP \mapsto \cP \in (\F(\X, \square), \tK)$ is a topological embedding map.
\end{thm}

This shows that the weak topology on $\Pi$ is a natural compact topology induced from the hyperspace $\F(\X, \square)$.

This paper is organized as follows.
In Section \ref{sec:prelim}, we describe some definitions and prepare some fundamental tools. A reader who is familiar with them can safely skip this section.
In Section \ref{sec:cpt_nbd}, we prove Theorem \ref{no_cpt_nbd} and Corollary \ref{sigma_cpt}.
A key technique is to make metric measure spaces near a given space by the $l_p$-product.
In Section \ref{sec:Baire}, we prove Theorem \ref{conc_baire} and some related properties.
A key tool is the box distance from the one-point space which is an invariant on $\X$.
In Section \ref{sec:contra}, we prove Theorem \ref{contra} by constructing explicit deformation retractions via the metric transformation.
In Section \ref{sec:geod}, we prove Theorem \ref{box_geod}. We will show that a midpoint between two spaces is given by the limit of the sequence of explicit spaces.
In Section \ref{sec:loc_path_conn}, we prove Theorem \ref{loc_path_conn} and Corollary \ref{Peano}. For the concentration and weak topologies, it is possible to create a good continuous path in a small ball instead of geodesics.
In Section \ref{sec:hypersp}, we prove Theorem \ref{emb} and describe the relation between the weak topology and the hyperspace theory.

Following this paper, we also study the scale-change action on the space of metric measure spaces in \cite{bundle}.
In \cite{bundle}, we have discovered the following surprising facts.
\begin{itemize}
\item $\X$ is not homeomorphic to a cone over the quotient space of the scale-change action.
\item This action induces a nontrivial and locally trivial principal bundle structure on $\X \setminus \{*\}$, where $*$ is a one-point metric measure space which is only one fixed point.
\end{itemize}
Moreover, a similar statement has been obtained for the space $\Pi$ of pyramids.

\section{Preliminaries}\label{sec:prelim}
In this section, we describe the definitions and some properties of metric measure space, the box distance, the observable distance, pyramid, and the weak topology. We use most of these notions along \cite{MMG}. As for more details, we refer to \cite{MMG} and \cite{Grmv}*{Chapter 3$\frac{1}{2}_+$}.

\subsection{Metric measure spaces}
Let $(X, d_X)$ be a complete separable metric space and $\mu_X$ a Borel probability measure on $X$. We call the triple $(X, d_X, \mu_X)$ a {\it metric measure space}, or an {\it mm-space} for short. We sometimes say that $X$ is an mm-space, in which case the metric and the measure of $X$ are respectively indicated by $d_X$ and $\mu_X$.

\begin{dfn}[mm-Isomorphism]
Two mm-spaces $X$ and $Y$ are said to be {\it mm-isomorphic} to each other if there exists an isometry $f \colon \supp{\mu_X} \to \supp{\mu_Y}$ such that $f_* \mu_X = \mu_Y$, where $f_* \mu_X$ is the push-forward measure of $\mu_X$ by $f$. Such an isometry $f$ is called an {\it mm-isomorphism}. Denote by $\mathcal{X}$ the set of mm-isomorphism classes of mm-spaces.
\end{dfn}

Note that an mm-space $X$ is mm-isomorphic to $(\supp{\mu_X}, d_X , \mu_X)$. We assume that an mm-space $X$ satisfies
\begin{equation*}
X = \supp{\mu_X}
\end{equation*}
unless otherwise stated. We denote by $*$ the one-point mm-space with trivial metric and Dirac measure.

\begin{dfn}[Lipschitz order]
Let $X$ and $Y$ be two mm-spaces. We say that $X$ ({\it Lipschitz}) {\it dominates} $Y$ and write $Y \prec X$ if there exists a $1$-Lipschitz map $f \colon X \to Y$ satisfying $f_* \mu_X = \mu_Y$. We call the relation $\prec$ on $\X$ the {\it Lipschitz order}.
\end{dfn}

The Lipschitz order $\prec$ is a partial order relation on $\X$.

\subsection{Box distance and observable distance}
For a subset $A$ of a metric space $(X, d_X)$ and for a real number $r > 0$, we set
\[
U_r(A) :=  \{x \in X \mid d_X(x, A) < r\},
\]
where $d_X(x, A) := \inf_{a \in A} d_X(x, a)$.

\begin{dfn}[Prokhorov distance]
The {\it Prokhorov distance} $\prok(\mu, \nu)$ between two Borel probability measures $\mu$ and $\nu$ on a metric space $X$ is defined to be the infimum of $\varepsilon > 0$ satisfying
\[
\mu(U_\varepsilon(A)) \geq \nu(A) - \varepsilon
\]
for any Borel subset $A \subset X$.
\end{dfn}

The Prokhorov metric $\prok$ is a metrization of the weak convergence of Borel probability measures on $X$ provided that $X$ is a separable metric space.

\begin{dfn}[Ky Fan metric]
Let $(X, \mu)$ be a measure space and $(Y, d_Y)$ a metric space. For two $\mu$-measurable maps $f,g \colon X \to Y$, we define $\kf^\mu (f, g)$ to be the infimum of $\varepsilon \geq 0$ satisfying
\begin{equation*}
\mu(\{x \in X \mid d_Y(f(x),g(x)) > \varepsilon \}) \leq \varepsilon.
\end{equation*}
The function $\kf^\mu$ is a metric on the set of $\mu$-measurable maps from $X$ to $Y$ by identifying two maps if they are equal to each other $\mu$-almost everywhere. We call $\kf^\mu$ the {\it Ky Fan metric}.
\end{dfn}

\begin{lem}[\cite{MMG}*{Lemma 1.26}]\label{prok_kf}
Let $X$ be a topological space with a Borel probability measure $\mu$ and $Y$ a metric space. For any two Borel measurable maps $f, g \colon X \to Y$, we have
\begin{equation*}
\prok(f_*\mu, g_* \mu) \leq \kf^\mu (f, g).
\end{equation*}
\end{lem}

\begin{dfn}[Parameter]
Let $I := [0,1)$ and let $X$ be an mm-space. A map $\varphi \colon I \to X$ is called a {\it parameter} of $X$ if $\varphi$ is a Borel measurable map such that
\begin{equation*}
\varphi_\ast \mathcal{L}^1 = \mu_X,
\end{equation*}
where $\mathcal{L}^1$ is the one-dimensional Lebesgue measure on $I$.
\end{dfn}

Note that any mm-space has a parameter (see \cite{MMG}*{Lemma 4.2}).

\begin{dfn}[Box distance]
We define the {\it box distance} $\square(X, Y)$ between two mm-spaces $X$ and $Y$ to be the infimum of $\varepsilon \geq 0$ satisfying that there exist parameters $\varphi \colon I \to X$, $\psi \colon I \to Y$, and a Borel subset $I_0 \subset I$ with $\mathcal{L}^1(I_0) \geq 1 - \varepsilon$ such that
\begin{equation*}
|d_X(\varphi(s), \varphi(t)) - d_Y(\psi(s), \psi(t))| \leq \varepsilon
\end{equation*}
for any $s,t \in I_0$.
\end{dfn}

We remark that $\square(X, Y) < 1$ for any $X, Y$ but can be as close to 1 as desired.

\begin{thm}[\cite{MMG}*{Theorem 4.10}]
The box distance function $\square$ is a complete separable metric on $\mathcal{X}$.
\end{thm}

Various distances equivalent to the box distance are defined and studied, for example, the Gromov-Prokhorov distance introduced by Greven-Pfaffelhuber-Winter \cite{GPW}.

\begin{thm}[\cite{Lohr}*{Theorem 3.1}, \cite{MMG}*{Remark 4.16}]\label{box_gp}
For any two mm-spaces $X$ and $Y$, we have
\[
\square(X, Y) = \gp((X, 2d_X, \mu_X), (Y, 2d_Y, \mu_Y)),
\]
where $\gp(X, Y)$ is the Gromov-Prokhorov metric defined to be the infimum of $\prok(\mu_X, \mu_Y)$ for all metrics on the disjoint union of $X$ and $Y$ that are extensions of $d_X$ and $d_Y$. In particular,
\[
\gp(X, Y) \leq \square(X, Y) \leq 2\gp(X, Y).
\]
\end{thm}

The topology induced from the box distance has historically various names, for example, the weak-Gromov topology. However we call it simply the {\it box topology} in this paper.

The following lemma is useful to calculate the box distance.

\begin{lem}[\cite{N}*{Theorem 1.1}]\label{box_opt}
Let $X$ and $Y$ be two mm-spaces. Then
\[
\square(X, Y) = \min_{\pi \in \Pi(\mu_X, \mu_Y)} \min_{S \subset X \times Y} \max\{\dis{S}, 1-\pi(S)\},
\]
where $\Pi(\mu_X, \mu_Y)$ is the set of couplings between $\mu_X$ and $\mu_Y$, and
\[
\dis{S} := \sup{\left\{|d_X(x, x') - d_Y(y, y')| : (x, y), (x', y') \in S \right\}}
\]
for a Borel subset $S \subset X \times Y$, which is called the {\it distortion} of $S$.
\end{lem}

\begin{cor}\label{box_1pt}
For any mm-space $X$,
\[
\square(X, *) = \min_{A \subset X} \max{\{\diam{A}, 1-\mu_X(A) \}}.
\]
In particular, if $Y \prec X$, then $\square(Y, *) \leq \square(X, *)$.
\end{cor}

Any mm-space can be approximated by a finite mm-space. Here, finite means having finitely many points.

\begin{prop}[\cite{MMG}*{Proposition 4.20}]\label{mmg4.20}
Let $X$ be an mm-space and let $\ep>0$. There exists a finite mm-space $\dot{X}$ such that $\square(X, \dot{X}) < \ep$.
\end{prop}

Given an mm-space $X$ and a parameter $\varphi \colon I \to X$ of $X$, we set
\[
\varphi^* \Lip_1(X) := \{ f \circ \varphi \mid f \colon X \to \R \text{ is $1$-Lipschitz} \},
\]
which consists of Borel measurable functions on $I$.

\begin{dfn}[Observable distance]
We define the {\it observable distance} $\conc(X, Y)$ between two mm-spaces $X$ and $Y$ by
\begin{equation*}
\conc(X, Y) := \inf_{\varphi, \psi} \haus(\varphi^* \Lip_1(X), \psi^* \Lip_1(Y)),
\end{equation*}
where $\varphi \colon I \to X$ and $\psi \colon I \to Y$ run over all parameters of $X$ and $Y$ respectively, and $\haus$ is the Hausdorff distance with respect to the metric $\kf^{\mathcal{L}^1}$.
\end{dfn}

\begin{thm}[\cite{MMG}*{Proposition 5.5 and Theorem 5.13}]
The observable distance function $\conc$ is a metric on $\mathcal{X}$. Moreover, for any two mm-spaces $X$ and $Y$,
\[
\conc(X, Y) \leq \square(X, Y).
\]
\end{thm}

We call the topology on $\X$ induced from $\conc$ the {\it concentration topology}.
We say that a sequence $\{X_n\}_{n=1}^\infty$ of mm-spaces  {\it concentrates} to an mm-space $X$ if $X_n$ $\conc$-converges to $X$ as $n \to \infty$. Since the concentration topology is coarser than the box topology, $(\X, \conc)$ is separable.

\begin{ex}
Let $S^n(1)$ be the $n$-dimensional unit sphere in $\R^{n+1}$ with the standard Riemannian structure. The sequence $\{S^n(1)\}_{n=1}^\infty$ is a typical example of concentrated sequences without any $\square$-convergent subsequence (see \cite{MMG}*{Corollary 5.20} or Lemma \ref{sphere_box}). $\{S^n(1)\}_{n=1}^\infty$ concentrates to the one-point mm-space $*$ as $n \to \infty$.

Furthermore, we consider the mm-spaces
\[
X_n := \prod_{k=1}^n S^k(1), \ n=1,2,\ldots,
\]
with the natural Riemannian product structure. The sequence $\{X_n\}_{n=1}^\infty$ is $\conc$-Cauchy but does not concentrate to any mm-space (see \cite{MMG}*{Example 7.36} and \cite{KY}). In particular, $(\X, \conc)$ is not complete.
\end{ex}

Denote by $\bar{\X}$ the completion of $(\X, \conc)$.

\subsection{Pyramid}

\begin{dfn}[Pyramid] \label{Py:dfn}
A subset $\cP \subset \X$ is called a {\it pyramid} if it satisfies the following {\rm(1) -- (3)}.
\begin{enumerate}
\item If $X \in \cP$ and if $Y \prec X$, then $Y \in \cP$.
\item For any $Y, Y' \in \cP$, there exists $X \in \cP$ such that $Y \prec X$ and $Y' \prec X$.
\item $\cP$ is nonempty and $\square$-closed.
\end{enumerate}
We denote the set of all pyramids by $\Pi$. Note that Gromov's definition of a pyramid is only by (1) and (2). The condition (3) is added in \cite{MMG}.

For an mm-space $X$, we define
\begin{equation*}
\cP_X := \left\{Y \in \X \midd Y \prec X \right\},
\end{equation*}
which is a pyramid. We call $\cP_X$ the {\it pyramid associated with $X$}.
\end{dfn}

We observe that $Y \prec X$ if and only if $\cP_Y \subset \cP_X$. Note that $\X$ itself is a pyramid.

We define the weak convergence of pyramids as follows. This is exactly the Kuratowski-Painlev\'e convergence as closed subsets of $(\X, \square)$ (see Definition \ref{KPconv}).

\begin{dfn}[Weak convergence]
Let $\cP$ and $\cP_n$, $n = 1, 2, \ldots$, be pyramids. We say that $\cP_n$ {\it converges weakly to} $\cP$ as $n \to \infty$ if the following (1) and (2) are both satisfied.
\begin{enumerate}
\item For any mm-space $X \in \cP$, we have
\begin{equation*}
\lim_{n \to \infty} \square(X, \cP_n) = 0.
\end{equation*}
\item For any mm-space $X \in \X \setminus \cP$, we have
\begin{equation*}
\liminf_{n \to \infty} \square(X, \cP_n) > 0.
\end{equation*}
\end{enumerate}
\end{dfn}

\begin{thm}[\cite{MMG}*{Section 6}]\label{Py:thm}
There exists a metric $\rho$ on $\Pi$ such that the following {\rm (1) -- (4)} hold.
\begin{enumerate}
\item $\rho$ is compatible with weak convergence.
\item $\Pi$ is $\rho$-compact.
\item The map
$\iota \colon \X \ni X \mapsto \cP_X \in \Pi$
is a $1$-Lipschitz topological embedding map with respect to
$\conc$ and $\rho$.
\item $\iota(\X)$ is $\rho$-dense in $\Pi$.
\end{enumerate}
\end{thm}
In particular, $(\Pi,\rho)$ is a compactification of $(\X,\conc)$. We often identify $X$ with $\cP_X$, and we say that a sequence of mm-spaces {\it converges weakly} to a pyramid if the associated pyramid converges weakly. In a minor abuse of notation, we use $\X$ as the image $\iota(\X)$ in $\Pi$.

\begin{rem}
One of constructions of the metric $\rho$ is as follows:
\[
\rho(\cP, \cP') := \sum_{k=1}^\infty \frac{1}{2^{k+2}k} \haus(\cP \cap \X(k,k), \cP'\cap \X(k,k)),
\]
where $\haus$ is the Hausdorff metric with respect to $\square$ and
\[
\X(N,R) := \left\{ (\R^N, \|\cdot\|_\infty, \mu) \midd \begin{array}{l} \mu \text{ is a Borel probability measure on } \R^N \text{ such that} \\ \supp\mu \text{ is contained in the closed $R$-ball centered at 0.} \end{array} \right\}.
\]
Note that $\X(N,R)$ is a $\square$-compact subset of $\X$.
\end{rem}

$(\Pi, \rho)$ is also a compactification of the completion $\bar{\X}$ of $(\X, \conc)$.

\begin{thm}[\cite{MMG}*{Theorem 7.27}]
The natural extension $\iota \colon \bar{\X} \to \Pi$ of the $1$-Lipschitz map $\iota \colon \X \ni X \mapsto \cP_X \in \Pi$ is a topological embedding map.
\end{thm}

We use $\bar{\X}$ as the image $\iota(\bar{\X})$ in $\Pi$ similar to $\X$.

The following proposition, which follows from the definition of the weak convergence directly, will be frequently used in this paper.

\begin{prop}\label{mmg6.2}
If a sequence $\{X_n\}_{n=1}^\infty$ of mm-spaces concentrates to an mm-space $X$ as $n \to \infty$, then there exists a sequence $\{Y_n\}_{n=1}^\infty$ of mm-spaces $\square$-converging to $X$ with $Y_n \prec X_n$ for every $n$.
\end{prop}

\begin{lem}[\cite{MMG}*{Lemma 7.14}]\label{approximation}
For any pyramid $\cP$, there exists a sequence $\{Y_m\}_{m = 1}^\infty$ of mm-spaces such that
\begin{equation*}
Y_1 \prec Y_2 \prec \cdots \prec Y_m \prec \cdots \quad \text{ and } \quad \overline{\bigcup_{m = 1}^\infty \cP_{Y_m}}^{\, \square} = \cP.
\end{equation*}
\end{lem}

Such a sequence $\{Y_m\}_{m=1}^\infty$ is called an {\it approximation of $\cP$}. We see that $Y_m$ converges weakly to $\cP$ as $m \to \infty$ and that $Y_m \in \cP$ for all $m$.

\begin{ex}[Virtual infinite-dimensional Gaussian space]\label{Gauss}
Let $\lambda$ be a positive real number. The $n$-dimensional Euclidean space $(\R^n,\|\cdot\|)$ with the $n$-dimensional centered Gaussian measure $\gamma^n_{\lambda^2}$ on $\R^n$ of variance $\lambda^2$ is called the {\it $n$-dimensional Gaussian space with variance $\lambda^2$}, write $\Gamma^n_{\lambda^2}$.
The natural projections from $\R^{n+1}$ to $\R^n$, $n=1,2,\ldots$, imply
\begin{equation*}
\Gamma^1_{\lambda^2} \prec \Gamma^2_{\lambda^2} \prec \cdots \prec \Gamma^n_{\lambda^2} \prec \cdots
\end{equation*}
and $\{\Gamma^n_{\lambda^2}\}_{n=1}^\infty$ converges weakly to the pyramid
\begin{equation*}
\cP_{\Gamma^\infty_{\lambda^2}} := \overline{\bigcup_{n=1}^\infty \cP_{\Gamma^n_{\lambda^2}}}^{\, \square}
\end{equation*}
as $n \to \infty$. We call $\cP_{\Gamma^\infty_{\lambda^2}}$ the {\it virtual infinite-dimensional Gaussian space with variance $\lambda^2$}.
We remark that $\cP_{\Gamma^\infty_{\lambda^2}}$ is neither in $\X$ nor in the completion $\bar{\X}$ (see \cite{MMG}*{Corollary 7.42}).
\end{ex}

\subsection{Observable diameter and metric transformation}

The observable diameter is one of the most fundamental invariants of an mm-space and a pyramid.

\begin{dfn}[Partial and observable diameter]
Let $X$ be an mm-space. For a real number $\alpha$, we define the {\it partial diameter} $\diam(X;\alpha) = \diam(\mu_X;\alpha)$ of $X$ to be the infimum of $\diam{A}$, where $A \subset X$ runs over all Borel subsets with $\mu_X (A) \geq \alpha$ and $\diam{A}$ denotes the diameter of $A$. For a real number $\kappa > 0$, we define the {\it observable diameter} of $X$ by
\[
\OD(X;-\kappa):= \sup\left\{\diam(f_* {\mu_X};1-\kappa) \midd f \colon X \to \R \text{ is $1$-Lipschitz}\right\} (<+\infty).
\]
Moreover, for a real number $\kappa > 0$, we define the {\it observable diameter of a pyramid} $\cP$ by
\[
\OD(\cP;-\kappa):= \lim_{\delta \to 0+}\sup_{X \in \cP} \OD(X;-(\kappa+\delta)) (\leq +\infty).
\]
\end{dfn}

The observable diameter for mm-spaces is an invariant under mm-isomorphism. Note that
\[
\OD(\cP_X;-\kappa) = \OD(X;-\kappa)
\]
for any $\kappa > 0$ and that $\OD(\cP;-\kappa)$ is monotone non-increasing and right-continuous in $\kappa > 0$. Moreover, we define
\[
\OD(\cP) := \inf_{\kappa > 0} \max\{\OD(\cP;-\kappa), \kappa\}
\]
for any pyramid $\cP$ and $\OD(X) := \OD(\cP_X)$ for any mm-space $X$. It is easy to see that
\[
\OD(\cP) = \sup_{X \in \cP} \OD(X).
\]

\begin{thm}[\cite{OS}*{Theorem 1.1}, Limit formula for observable diameter]\label{lim_form}
Let $\cP$ and $\cP_n$, $n = 1, 2, \ldots$, be pyramids. If $\cP_n$ converges weakly to $\cP$ as $n \to \infty$, then
\begin{align*}
\OD(\cP;-\kappa) & = \lim_{\ep\to 0+} \liminf_{n\to\infty} \OD(\cP_n;-(\kappa+\ep)) \\
& = \lim_{\ep\to 0+} \limsup_{n\to\infty} \OD(\cP_n;-(\kappa+\ep))
\end{align*}
for any $\kappa > 0$.
\end{thm}

\begin{thm}[\cite{OS}*{Corollary 5.8}]\label{Levyfam}
Let $\cP_n$, $n = 1,2,\ldots$, be pyramids. Then the following {\rm (1)} -- {\rm (3)} are
equivalent to each other.
\begin{enumerate}
\item $\cP_n$ converges weakly to the one-point mm-space $*$ as $n\to\infty$.
\item $\lim_{n\to\infty} \OD(\cP_n;-\kappa) = 0$ for any $\kappa > 0$.
\item $\lim_{n\to\infty} \OD(\cP_n)= 0$.
\end{enumerate}
\end{thm}

\begin{ex}[cf.~\cite{OS}*{Example 3.13}]\label{GaussOD}
The observable diameter of the virtual infinite-dimensional Gaussian space $\cP_{\Gamma^\infty_{\lambda^2}}$ with variance $\lambda^2$ is
\[
\OD(\cP_{\Gamma^\infty_{\lambda^2}};-\kappa) = \diam(\gamma^1_{\lambda^2};1-\kappa) = 2\lambda I^{-1}((1-\kappa)/2)
\]
for any $\kappa$ and $\lambda$ with $0 < \kappa < 1$ and $\lambda \geq 0$, where
\[
I(r) := \gamma^1_{1^2}([0, r]) = \frac{1}{\sqrt{2\pi}} \int_0^r \exp(-\frac{x^2}{2}) \, dx.
\]
Therefore $\cP_{\Gamma^\infty_{\lambda^2}}$ converges weakly to $*$ as $\lambda \to 0$.
\end{ex}

\begin{dfn}[Metric transformation]
A function $F \colon [0,+\infty) \to [0, +\infty)$ is a {\it metric preserving function} provided that $F \circ d_X$ is a metric on $X$ for any metric space $(X, d_X)$. For a metric preserving function $F$, we define the {\it metric transformation} of an mm-space $X$ and of a pyramid $\cP$ by
\[
F(X) := (X, F \circ d_X, \mu_X) \quad \text{and} \quad F(\cP) := \overline{\bigcup_{X \in \cP} \cP_{F(X)}}^{\, \square}.
\]
\end{dfn}

If a metric preserving function $F$ is continuous, the topologies of $F(X)$ and $X$ coincide. In addition, if $F$ is nondecreasing, $F(\cP)$ is a pyramid for any pyramid $\cP$ and
\[
F(\cP_X) = \cP_{F(X)}
\]
holds for every mm-space $X$. Note that if $\{Y_m\}_{m=1}^\infty$ is an approximation of a pyramid $\cP$, then $\{F(Y_m)\}_{m=1}^\infty$ is an approximation of $F(\cP)$.

Let $F(s) := ts$ for $t > 0$, which is a continuous nondecreasing metric preserving function. We denote $F(X)$ and $F(\cP)$ by $tX$ and $t\cP$, respectively. Note that
\[
t\cP = \left\{tX \midd X \in \cP \right\},
\]
which is considered classically.

\begin{lem}\label{ODtrans}
Let $F$ be a continuous nondecreasing metric preserving function. Then we have
\[
\OD(F(\cP); -2\kappa) \leq 4F(\OD(\cP;-\kappa))
\]
for any pyramid $\cP$ and any $\kappa > 0$, where we agree that $F(+\infty) = \sup_{s > 0} F(s)$.
\end{lem}

\begin{proof}
For any mm-space $X$ and any $\kappa > 0$, we already have obtained the same estimate
\[
\OD(F(X); -2\kappa) \leq 4F(\OD(X;-\kappa))
\]
in \cite{prod}*{Lemma 3.22}. Using this, we check for a given pyramid $\cP$. Let $\{Y_m\}_{m=1}^\infty$ be an approximation of $\cP$. By Theorem \ref{lim_form}, we have
\begin{align*}
\OD(F(\cP);-2\kappa) & = \lim_{\ep\to 0+} \liminf_{m\to\infty} \OD(F(Y_m);-2(\kappa+\ep)) \\
& \leq \lim_{\ep\to 0+} \liminf_{m\to\infty} 4F(\OD(Y_m;-(\kappa+\ep))) \\
& \leq \lim_{\ep\to 0+} 4F(\OD(\cP;-(\kappa+\ep))) \\
& \leq 4F(\OD(\cP;-\kappa)).
\end{align*}
The proof is completed.
\end{proof}

\begin{dfn}
Let $\cP$ and $\cP'$ be two pyramids and let $1 \leq p \leq +\infty$. We define the $l_p$-product of $\cP$ and $\cP'$ by
\[
\cP \times_p \cP' := \overline{\bigcup_{X \in \cP, Y \in \cP'} \cP_{X \times_p Y}}^{\, \square},
\]
where $X \times_p Y$ is the $l_p$-product space of two mm-spaces $X$ and $Y$.
\end{dfn}
Note that if $\{X_m\}_{m=1}^\infty$ and $\{Y_m\}_{m=1}^\infty$ are approximations of pyramids $\cP$ and $\cP'$, respectively, then $\{X_m\times_p Y_m\}_{m=1}^\infty$ is an approximation of $\cP \times_p \cP'$.

\section{No compact neighborhood in $\X$}\label{sec:cpt_nbd}
In this section, we prove Theorem \ref{no_cpt_nbd} and Corollary \ref{sigma_cpt}. For the box topology, we prove the following lemma which implies Theorem \ref{no_cpt_nbd} (in fact they are equivalent).

\begin{lem}\label{lem:tot_bdd}
Any neighborhood of an mm-space $X$ is not precompact with respect to the box topology.
\end{lem}

\begin{proof}
For any small $\ep > 0$, it is sufficient to prove that $U_{2\ep}(X) \subset (\X, \square)$ is not precompact. We construct a countable discrete net in $U_{2\ep}(X)$. There exists a finite mm-space $\dot{X}$ such that $\square(X, \dot{X}) < \ep$ by Proposition \ref{mmg4.20} and let $N := \#\dot{X}$, where $\#$ means the number of points.
We define mm-spaces $Y_n$, $n=1,2,\ldots$, as
\[
Y_n := \{1, 2, \ldots, (2N)^n\}
\]
with metric $d_{Y_n}(i, j) := \ep$ for $i\neq j$ and uniform probability measure $\mu_{Y_n} := (2N)^{-n}\sum_{i=1}^{(2N)^n}\delta_i$, and define $X_n := \dot{X} \times_\infty Y_n$. Let us prove that $\{X_n\}_{n=1}^\infty$ is a discrete net in $U_{2\ep}(X)$.

Since $\square(A \times_p B, A \times_p C) \leq \square(B,C)$ in general (see \cite{prod}*{Proposition 4.1}), we have
\[
\square(X_n, X) < \square(X_n, \dot{X})+\ep \leq \square(Y_n, *) + \ep \leq 2\ep.
\]
Thus $X_n \in U_{2\ep}(X)$ for any $n$. We next prove that
\[
\square(X_m, X_n) \geq \min\{\ep, \min_{x\neq x'} d_{\dot{X}}(x,x'), \frac{1}{2}\} =: \delta > 0
\]
for any $m\neq n$. Assume that $m > n$. By Lemma \ref{box_opt}, there exist a coupling measure $\pi \in \Pi(\mu_{X_m}, \mu_{X_n})$ and a closed set $S \subset X_m \times X_n$ such that
\[
\square(X_m, X_n) = \max\{\dis{S}, 1-\pi(S)\}.
\]
If there exist two pairs $(x, y), (x', y) \in S$ with $x \neq x'$ in $X_m$, then we have
\[
\square(X_m, X_n) \geq \dis{S} \geq d_{X_m}(x,x') \geq \delta.
\]
If not, then we have $\#S \leq \#X_n = N(2N)^n$ and
\[
\pi(S) \leq \frac{\#S}{\#Y_m}\max_{x\in\dot{X}}{\mu_{\dot{X}}(\{x\})} \leq \frac{N(2N)^n}{(2N)^m} \leq \frac{1}{2},
\]
which implies that $\square(X_m, X_n) \geq 1-\pi(S) \geq \frac{1}{2} \geq \delta$. Therefore $\{X_n\}_{n=1}^\infty$ is $\delta$-discrete. This completes the proof.
\end{proof}

In order to prove Theorem \ref{no_cpt_nbd} for the concentration topology, we start with recalling the following proposition.

\begin{prop}[cf.~\cite{E}*{Theorem 3.3.9}]\label{prop:loc_cpt}
Let $Y$ be a Hausdorff space and let $X \subset Y$ be a dense subset. If a point $x \in X$ has a compact neighborhood {\rm (}in the relative topology of $Y${\rm )}, then $x$ is an interior point of $X$. In particular, if $X$ is locally compact, then $X$ is open in $Y$.
\end{prop}

The $l_p$-product is useful again to construct a convergent sequence to a given mm-space or pyramid. The following proposition is known.

\begin{prop}[\cite{MMG}*{Proposition 7.32}]\label{mmg7.32}
Let $X$ and $Y$ be two mm-spaces. If $\OD(Y) < 1/2$, then we have
\[
\conc(X \times_p Y, X) \leq \OD(Y)
\]
for any $1 \leq p \leq +\infty$.
\end{prop}

This proposition can be generalized to pyramids as follows.

\begin{prop}\label{ODest}
Let $\cP$ and $\cP'$ be two pyramids. If $\OD(\cP') < 1/2$, then we have
\[
\rho(\cP \times_p \cP', \cP) \leq \OD(\cP')
\]
for any $1 \leq p \leq +\infty$.
\end{prop}

\begin{proof}
Let $\{X_m\}_{m=1}^\infty$ and $\{Y_m\}_{m=1}^\infty$ be approximations of $\cP$ and $\cP'$ respectively. For each $m$, we have
\[
\rho(\cP_{X_m \times_p Y_m}, \cP_{X_m}) \leq \conc(X_m \times_p Y_m, X_m) \leq \OD(Y_m) \leq \OD(\cP')
\]
by Proposition \ref{mmg7.32}. Thus we have
\[
\rho(\cP \times_p \cP', \cP) = \lim_{m\to\infty} \rho(\cP_{X_m \times_p Y_m}, \cP_{X_m}) \leq  \OD(\cP').
\]
The proof is completed.
\end{proof}

\begin{lem}
For given $X \in \X$ and $1 \leq p \leq +\infty$, the pyramid $\cP_\lambda$ for $\lambda > 0$ is defined by
\[
\cP_\lambda := \cP_X \times_p \cP_{\Gamma^\infty_{\lambda^2}}.
\]
Then $\cP_\lambda$ converges weakly to $X$  as $\lambda \to 0$.
\end{lem}

\begin{proof}
This follows directly from Theorem \ref{Levyfam}, Example \ref{GaussOD}, and Proposition \ref{ODest}.
\end{proof}

Since $\X$ and $\bar{\X}$ are downward-closed in $\Pi$ with respect to the inclusion, the pyramid $\cP_\lambda$ is neither in $\X$ nor in $\bar{\X}$ (see Example \ref{Gauss}). This leads to the following corollaries.

\begin{cor}\label{no_int}
Every $X \in \X$ is not an interior point of $\X$ with respect to the weak topology. Similarly, every $\bar{X} \in \bar{\X}$ is not an interior point of $\bar{\X}$ with respect to the weak topology.
\end{cor}

\begin{cor}
Both $\Pi \setminus \X$ and $\Pi \setminus \bar{\X}$ are dense in $\Pi$.
\end{cor}

\begin{proof}[Proof of Theorem \ref{no_cpt_nbd}]
With respect to the box topology, Lemma \ref{lem:tot_bdd} implies Theorem \ref{no_cpt_nbd}.
Proposition \ref{prop:loc_cpt} and Corollary \ref{no_int} mean that every $X \in \X$ has no compact neighborhood in the relative weak topology, that is, the concentration topology. The proof is completed.
\end{proof}

\begin{prop}[cf.~\cite{W}*{25B}]\label{prop:sigma-cpt}
If a topological space $X$ is Hausdorff, $\sigma$-compact, and Baire, then at least one point in $X$ has a compact neighborhood.
\end{prop}

\begin{proof}
Let $X$ be a Hausdorff, $\sigma$-compact, and Baire space. There exists countable family $\{K_i\}_{i=1}^\infty$ of compact subsets of $X$ such that $X = \bigcup_{i=1}^\infty K_i$. Since $X$ is Hausdorff, each $K_i$ is closed. Since $X$ is Baire, at least one of $\{K_i\}_{i=1}^\infty$ must be nonempty interior. This completes the proof.
\end{proof}

\begin{proof}[Proof of Corollary \ref{sigma_cpt}]
Since $(\X, \square)$ is complete metric space, $(\X, \square)$ is Hausdorff and Baire. Theorem \ref{no_cpt_nbd} and Proposition \ref{prop:sigma-cpt} together mean that $(\X, \square)$ is not $\sigma$-compact.
\end{proof}

\begin{rem}
Let $\{Y_m\}_{m=1}^\infty$ be an approximation of the pyramid $\X$. It holds that
\[
\X = \overline{\bigcup_{m=1}^\infty \cP_{Y_m}}^{\, \square}
\]
and each $\cP_{Y_m}$ is $\square$-compact (see \cite{KY}). Here, Corollary \ref{sigma_cpt} says that the $\square$-closure operation is essential, namely
\[
\bigcup_{m=1}^\infty \cP_{Y_m} \subsetneq \X.
\]
\end{rem}

The following is obtained by the same reason as Corollary \ref{sigma_cpt}.

\begin{cor}
The completion $\bar{\X}$ of $(\X, \conc)$ is not $\sigma$-compact.
\end{cor}

\section{$\X$ with the concentration topology is not a Baire space}\label{sec:Baire}

In this section, we prove Theorem \ref{conc_baire}. The key tool is the box distance from the one-point mm-space $*$, which is an invariant on $\X$ (see Corollary \ref{box_1pt}).

\begin{prop}\label{box_1pt_est}
For any mm-space $X$,
\[
\square(X, *) \geq 1 - \sup_{x\in X}{\mu_X(U_1(x))}.
\]
\end{prop}

\begin{proof}
Take any Borel subset $A$ of $X$ with $\diam{A} < 1$ and choose $x \in A$. Then $A \subset U_1(x)$ and
\[
1 - \mu_X(A) \geq 1 - \mu_X(U_1(x)).
\]
Thus we have
\[
\max{\{\diam{A}, 1 - \mu_X(A)\}} \geq 1 - \mu_X(U_1(x)) \geq 1 - \sup_{x \in X}\mu_X(U_1(x)).
\]
Corollary \ref{box_1pt} implies the desired inequality.
\end{proof}

\begin{lem}\label{sphere_box}
Let $S^n(1)$ be the $n$-dimensional unit sphere with the standard Riemannian structure. Then
\[
\lim_{n \to \infty} \square(S^n(1), *) = 1.
\]
\end{lem}

\begin{proof}
For any $x \in S^n(1)$, we have
\[
\lim_{n \to \infty}\mu_{S^n(1)}(U_1(x)) = \lim_{n \to \infty}\frac{\int_0^1 \sin^{n-1}{t} \, dt}{\int_0^\pi \sin^{n-1}{t} \, dt} = 0.
\]
Combining this and Proposition \ref{box_1pt_est} implies $\lim_{n \to \infty} \square(S^n(1), *) = 1$.
\end{proof}

\begin{prop}\label{lsc_box}
If a sequence $\{X_n\}_{n=1}^\infty$ of mm-spaces concentrates to an mm-space $X$, then
\[
\square(X, *) \leq \liminf_{n\to\infty} \square(X_n, *).
\]
\end{prop}

\begin{proof}
Since $\{X_n\}_{n=1}^\infty$ concentrates to $X$, there exists a sequence $\{Y_n\}_{n=1}^\infty$ of mm-spaces $\square$-converging to $X$ with $Y_n \prec X_n$ for every $n$, by Proposition \ref{mmg6.2}. Thus we have
\[
\liminf_{n\to\infty} \square(X_n, *) \geq \lim_{n\to\infty} \square(Y_n, *) = \square(X, *).
\]
The proof is completed.
\end{proof}

\begin{proof}[Proof of Theorem \ref{conc_baire}]
Let $\X^\delta$ be the set of all mm-spaces with $\square(X, *) \leq \delta$ for $\delta \geq 0$. By Proposition \ref{lsc_box}, the set $\X^\delta$ is closed with respect to the concentration topology. Since
\[
\X = \bigcup_{n=1}^\infty \X^{1-\frac{1}{n}},
\]
if $\X^\delta$ is nowhere dense for any $\delta \in [0, 1)$, then $\X$ is not a Baire space. We prove that $\X^\delta$ is nowhere dense. Take any mm-space $X \in \X^\delta$. It is sufficient to prove that $X$ is not an interior point of $\X^\delta$. Indeed, the product space $X \times_p S^n(1)$ concentrates to $X$ as $n \to \infty$ by Proposition \ref{mmg7.32} but
\[
\liminf_{n\to\infty}\square(X \times_p S^n(1), *) \geq \lim_{n\to\infty}\square(S^n(1), *) = 1.
\]
Thus $X$ is not an interior point of $\X^\delta$. The proof is completed.
\end{proof}

\begin{rem}
\begin{enumerate}
\item From the above proof, $\X$ with the concentration topology is meager (i.e., a countable union of nowhere dense subsets) in itself. Actually, this fact is stronger than non-Baire.
\item The Baire category theorem claims that $\X$ is not completely metrizable, equivalently, $\X$ is not a $G_\delta$ subset of $\Pi$. Namely, there is no complete metric giving the concentration topology.
\end{enumerate}
\end{rem}

\begin{cor}
$\X$ is meager and non-comeager in $\Pi$.
\end{cor}

\begin{proof}
We first prove that the subset $\X$ of $\Pi$ is meager. By the definition of the weak convergence and Lemma \ref{approximation}, the closure of $\X^\delta$ with respect to the weak topology is
\[
\Pi^\delta := \left\{\cP \in \Pi \midd \square(X, *) \leq \delta \text{ for any } X \in \cP \right\}.
\]
It is sufficient to prove that the interior of $\Pi^\delta$ is empty for every $\delta \in [0, 1)$. Actually, given a pyramid $\cP \in \Pi^\delta$, the product pyramid $\cP \times_p \cP_{S^n(1)}$ converges weakly to $\cP$ by Proposition \ref{ODest} and $S^n(1) \in \cP \times_p \cP_{S^n(1)}$. These imply that $\cP$ is not an interior point of $\Pi^\delta$. Thus the subset $\X$ of $\Pi$ is meager.

Suppose that $\X$ is comeager in $\Pi$. The complement $\Pi\setminus\X$ is meager and hence $\Pi$ is meager in itself. This is a contradiction. Thus $\X$ is not comeager in $\Pi$.
\end{proof}

\begin{rem}
$\X^\delta$ is properly included in $\Pi^\delta$ as a subset of $\Pi$. For example,
\[
\left\{X \in \X \midd \diam X \leq \delta\right\}
\]
is a pyramid in $\Pi^\delta$ but it is not $\square$-compact, so that it belongs to $\Pi\setminus\X$ (see \cite{KY}). Indeed, letting $X_n$, $n=1,2,\ldots$, as the $l_\infty$-product space of $n$ copies of the interval $[0,\delta]$, the sequence  $\{X_n\}_{n=1}^\infty$ has no $\square$-convergent subsequence (see \cite{MMG}*{Proposition 7.37}).
\end{rem}

\section{$\X$ and $\Pi$ are contractible}\label{sec:contra}

In this section, we prove Theorem \ref{contra} by giving explicit deformation retractions. From now on, for two given maps $f$ and $g$ on a space $X$, we use the notation $(f, g)$ as the map on $X$ defined by
\[
(f, g)(x) := (f(x), g(x)), \quad x \in X.
\]
For example, $(\id_X, \id_X)$ means the map $X \ni x \mapsto (x, x) \in X \times X$.

\begin{lem}\label{retract_mm}
The map $H \colon \X \times [0,1] \to \X$ defined by
\[
H(X, t) := tX
\]
for $X \in \X$ and $t \in [0, 1]$ is continuous with respect to both the box and concentration topologies, where we agree that $0X = *$ for any mm-space $X$.
\end{lem}

\begin{proof}
Since the discussions are parallel, we prove only for the box topology. We take any $\{(X_n, t_n)\}_{n=1}^\infty \subset \X \times [0,1]$ converging to $(X, t)$. Then we have
\[
\square(t_nX_n, tX) \leq \square(t_nX_n, t_nX) + \square(t_nX, tX) \leq \square(X_n, X) + \square(t_nX, tX)
\]
since the map $t \mapsto \square(tX, tY)$ is nondecreasing. It is sufficient to prove that $t_nX$ $\square$-converges to $tX$ as $t_n \to t$ for a fixed mm-space $X$. Let $\ep > 0$ be a positive real number. There exists a finite mm-space $\dot{X}$ such that $\square(X, \dot{X}) < \ep$ by Proposition \ref{mmg4.20}.
Then we have
\[
\square(t_n\dot{X}, t\dot{X}) \leq |t_n - t|\diam{\dot{X}}.
\]
Indeed, letting $\pi := (\id_{\dot{X}}, \id_{\dot{X}})_* \mu_{\dot{X}}$ and $S := \{(x,x) \ | \ x \in \dot{X}\}$, we have
\[
\pi(S) = 1 \quad \text{and} \quad \dis{S} = |t_n - t|\diam{\dot{X}}.
\]
Note that this is true even if $t_n = 0$ or $t = 0$. Thus we have
\[
\limsup_{n \to \infty} \square(t_nX, tX) \leq \limsup_{n \to \infty} \square(t_n\dot{X}, t\dot{X}) +2\ep \leq 2\ep.
\]
As $\ep \to 0$, we obtain the conclusion.
\end{proof}

\begin{lem}\label{retract_py}
The map $H \colon \Pi \times [0,1] \to \Pi$ defined by
\[
H(\cP, t) := F_t(\cP), \quad \text{where} \quad F_t(s) := \min\{s, \frac{t}{1-t}\},
\]
for $\cP \in \Pi$ and $t \in [0, 1]$ is continuous with respect to the weak topology, where we agree that $F_1(\cP) = \cP$ and $F_0(\cP) = *$ for any pyramid $\cP$.
\end{lem}

\begin{proof}
We take any $\{(\cP_n, t_n)\}_{n=1}^\infty \subset \Pi \times [0,1]$ converging to $(\cP, t)$.
The main result of \cite{comts} implies that $F_{t_n}(\cP_n)$ converges weakly to $F_t(\cP)$ if $t > 0$ (see \cite{comts}*{Corollary 1.5}). We check only that if $t_n \to 0$, then $F_{t_n}(\cP_n)$ converges weakly to $*$ as $n \to \infty$. By Lemma \ref{ODtrans}, for any $\kappa > 0$, we have
\[
\OD(F_{t_n}(\cP_n); -2\kappa) \leq 4F_{t_n}(\OD(\cP_n; -\kappa)) \leq \frac{4t_n}{1-t_n} \to 0
\]
as $n \to \infty$. Therefore $F_{t_n}(\cP_n)$ converges weakly to $*$ as $n \to \infty$ by Theorem \ref{Levyfam}. The proof is completed.
\end{proof}

\begin{rem}
The map $(\cP, t) \mapsto t\cP$ is discontinuous, in fact, $t\X = \X$ for any $t > 0$. One reason for this is that the function $s \mapsto ts$ does not converge uniformly to 0 as $t \to 0$. On the other hand, the map $(X, t) \mapsto F_t(X)$, where $F_t$ in above lemma, is also continuous with respect to both the box and concentration topologies.
\end{rem}

\begin{proof}[Proof of Theorem \ref{contra}]
The maps in Lemmas \ref{retract_mm} and \ref{retract_py} are deformation retractions of $\X$ and $\Pi$ onto $\{*\}$, respectively. Therefore these are contractible.
\end{proof}

\section{$(\X, \square)$ is a geodesic space}\label{sec:geod}

The aim of this section is to prove Theorem \ref{box_geod}. We prepare several tools due to the optimal transport theory.

\begin{dfn}[$\ep$-Subtransport plan]
Let $\mu$ and $\nu$ be two Borel probability measures on a metric space $X$. A Borel measure $\pi$ on $X \times X$ is called a {\it subtransport plan} between $\mu$ and $\nu$ provided that ${\pr_0}_* \pi \leq \mu$ and ${\pr_1}_* \pi \leq \nu$, where $\pr_i$, $i=0, 1$, is the projection given by $(x_0, x_1) \mapsto x_i$. For a subtransport plan $\pi$, the {\it deficiency} of $\pi$ is defined to be
\[
\df{\pi} := 1 - \pi(X \times X).
\]
A subtransport plan $\pi$ is called an {\it $\ep$-subtransport plan} if it satisfies
\[
\supp{\pi} \subset \{ (x, x') \in X \times X \mid d_X(x, x') \leq \ep \}.
\]
\end{dfn}

\begin{thm}[Strassen's theorem \cite{V}*{Corollary 1.28}]\label{Strassen}
For any two Borel probability measures $\mu$ and $\nu$ on a complete separable metric space $X$, we have
\[
\prok(\mu, \nu) = \inf{\left\{ \ep > 0 \midd \begin{array}{l} \text{There exists an } \ep\text{-subtransport plan } \pi \\ \text{between } \mu \text{ and } \nu \text{ with } \df{\pi} \leq \ep \end{array} \right\}}.
\]
\end{thm}

Let $\Cb(X)$ be the set of all continuous bounded real-valued functions on a metric space $X$, which is a Banach space with the supremum norm $\|\cdot\|_\infty$. It is well-known that the map
\[
X \ni x \mapsto d_X(x, x') - d_X(\bar{x}, x') \in \Cb(X),
\]
where $\bar{x}$ is a fixed point in $X$, is isometric. This map is called the Kuratowski embedding.

\begin{proof}[Proof of Theorem \ref{box_geod}]
Take any two mm-spaces $X_0$ and $X_1$. We construct a midpoint $X_{\frac{1}{2}}$ between $X_0$ and $X_1$, that is,
\[
\square(X_0, X_{\frac{1}{2}}) = \square(X_1, X_{\frac{1}{2}}) = \frac{1}{2} \square(X_0, X_1)
\]
(see \cite{BBI}*{Theorem 2.4.16}). Let $r_n := \square(X_0, X_1) + n^{-1}$. Since
\[
\gp(2X_0, 2X_1) = \square(X_0, X_1) < r_n,
\]
there exists a complete separable metric space $Z_n$ such that both $2X_0$ and $2X_1$ are embedded in $Z_n$ isometrically and
\[
\prok^{Z_n}(\mu_{X_0}, \mu_{X_1}) < r_n.
\]
Moreover, by the Kuratowski embedding, $Z_n$ can be assumed to be a Banach space with norm $\|\cdot\|$. By Strassen's theorem, there exists an $r_n$-subtransport plan $\pi_n$ between $\mu_{X_0}$ and $\mu_{X_1}$ over $(Z_n, \|\cdot\|)$ with $\df{\pi_n} \leq r_n$. We define a map $M \colon Z_n \times Z_n \to Z_n$ by
\[
M(x_0, x_1) := \frac{x_0 + x_1}{2}
\]
and define a probability measure $\mu_{\frac{1}{2}, n}$ on $Z_n$ by
\begin{align*}
\mu_{\frac{1}{2}, n} & := M_* \pi_n  + \frac{1}{2}(\mu_{X_0}-{\pr_0}_* \pi_n)+\frac{1}{2}(\mu_{X_1}-{\pr_1}_* \pi_n).
\end{align*}
Let us prove
\[
\prok^{Z_n}(\mu_{X_i}, \mu_{\frac{1}{2}, n}) \leq \frac{1}{2}r_n, \quad i=0,1.
\]

We first find a subtransport plan between $\mu_{X_0}$ and $\mu_{\frac{1}{2}, n}$. We define a measure $\pi_{0, n}$ on $Z_n \times Z_n$ by
\[
\pi_{0,n} := (\pr_0, M)_* \pi_n + \frac{1}{2} (\id_{Z_n}, \id_{Z_n})_* (\mu_{X_0} - {\pr_0}_* \pi_n).
\]
The measure $\pi_{0,n}$ is a subtransport plan between $\mu_{X_0}$ and $\mu_{\frac{1}{2},n}$. Indeed,
\begin{align*}
{\pr_0}_* \pi_{0,n} & = {\pr_0}_* \pi_n + \frac{1}{2} (\mu_{X_0} - {\pr_0}_* \pi_n) \leq \mu_{X_0}, \\
{\pr_1}_* \pi_{0,n} & = M_* \pi_n + \frac{1}{2} (\mu_{X_0} - {\pr_0}_* \pi_n) \leq \mu_{\frac{1}{2},n}.
\end{align*}
Moreover, we have
\[
\pi_{0,n}(Z_n \times Z_n) = \pi_n(Z_n \times Z_n) + \frac{1}{2}(1 - \pi_n(Z_n \times Z_n)) \geq 1 - \frac{1}{2}r_n.
\]
We verify that $\pi_{0,n}$ is a $(1/2)r_n$-subtransport plan. Take any $(x, y) \in \supp{\pi_{0,n}}$ with $x \neq y$. There exists $(x_0, x_1) \in \supp{\pi_n}$ such that
\[
x = x_0 \quad \text{ and } \quad y = \frac{x_0 + x_1}{2}.
\]
Then we have
\[
\|x - y\| = \frac{1}{2}\|x_0 - x_1\| \leq \frac{1}{2} r_n.
\]
Therefore $\pi_{0,n}$ is a $(1/2)r_n$-subtransport plan with $\df{\pi_{0,n}} \leq (1/2)r_n$, which implies that
\[
\prok^{Z_n}(\mu_{X_0}, \mu_{\frac{1}{2},n}) \leq \frac{1}{2}r_n
\]
by Strassen's theorem again. Similarly, letting
\[
\pi_{1,n} := (M, \pr_1)_* \pi_n + \frac{1}{2} (\id_{Z_n}, \id_{Z_n})_* (\mu_{X_1} - {\pr_1}_* \pi_n),
\]
the measure $\pi_{1,n}$ is a $(1/2)r_n$-subtransport plan between $\mu_{\frac{1}{2},n}$ and $\mu_{X_1}$ with $\df{\pi_{1,n}} \leq (1/2)r_n$ and hence
\[
\prok^{Z_n}(\mu_{X_1}, \mu_{\frac{1}{2},n}) \leq \frac{1}{2}r_n.
\]

Defining an mm-space
\[
X_{\frac{1}{2},n} := (Z_n, \frac{1}{2}\|\cdot\|, \mu_{\frac{1}{2}, n}),
\]
this satisfies
\[
\square(X_i, X_{\frac{1}{2},n}) = \gp(2X_i, 2X_{\frac{1}{2}, n}) \leq \prok^{Z_n}(\mu_{X_i}, \mu_{\frac{1}{2}, n}) \leq\frac{1}{2}r_n, \quad i=0,1.
\]
We prove that $\{X_{\frac{1}{2}, n}\}_{n=1}^\infty$ is precompact with respect to the box topology. It is sufficient to prove that for any $\ep > 0$ there exists a positive number $\Delta(\ep)$ such that for any $n$ we have a finite subset $\cN_n$ of $X_{\frac{1}{2}, n}$ with
\[
\mu_{\frac{1}{2}, n} (U_{\ep}(\cN_n)) \geq 1 - \ep, \quad \#\cN_n \leq \Delta(\ep), \quad \text{and} \quad \diam{\cN_n} \leq \Delta(\ep)
\]
(see \cite{MMG}*{Lemma 4.28}). Take a sufficiently small $\ep > 0$ such that $1 - r_n -2\ep > 0$ for every sufficiently large $n$, here $r_n \to \square(X_0, X_1) < 1$ as $n \to \infty$. There exist a finite subset $\cN_0$ of $X_0$ and a finite subset $\cN_1$ of $X_1$ such that
\[
\mu_{X_0}(U_\ep(\cN_0)) \geq 1-\ep \quad \text{ and } \quad \mu_{X_1}(U_\ep(\cN_1)) \geq 1-\ep.
\]
Let $\cZ_0$ and $\cZ_1$ be the images of $\cN_0$ and $\cN_1$ by the embeddings to $Z_n$, respectively, and define
\[
\cN_n := \cZ_0 \cup M(\cZ_0\times \cZ_1) \cup \cZ_1 \subset Z_n.
\]
Note that $\# \cN_n \leq \#\cN_0 \cdot \#\cN_1 + \#\cN_0 + \#\cN_1 $.
We have
\begin{align*}
& \pi_n(U_{2\ep}(\cZ_0) \times U_{2\ep}(\cZ_1)) \\
& \geq \pi_n(Z_n \times Z_n) - {\pr_0}_*\pi_n(Z_n \setminus U_{2\ep}(\cZ_0)) - {\pr_1}_*\pi_n(Z_n \setminus U_{2\ep}(\cZ_1)) \\
& \geq 1 - r_n - \mu_{X_0}(X_0 \setminus U_{\ep}(\cN_0)) - \mu_{X_1}(X_1 \setminus U_{\ep}(\cN_1)) \geq 1 - r_n -2\ep > 0,
\end{align*}
which implies that there exists a pair $(\bar{x}, \bar{y}) \in \supp\pi_n \cap (U_{2\ep}(\cZ_0) \times U_{2\ep}(\cZ_1))$.
Moreover,
\[
\|M(x_0, x_1) - M(\bar{x}, \bar{y})\| \leq \frac{1}{2} \|x_0 - \bar{x}\| + \frac{1}{2}\|x_1 - \bar{y}\| \leq \frac{1}{2} \diam{\cZ_0} + \frac{1}{2} \diam{\cZ_1} + 2\ep
\]
for every $(x_0, x_1) \in \cZ_0 \times \cZ_1$. Combining these implies that
\begin{align*}
\diam{\cN_n} & \leq \frac{3}{2}\diam{\cZ_0} + \frac{3}{2}\diam{\cZ_1} + \|\bar{x} - \bar{y}\| + 4\ep \\
& \leq 3\diam{\cN_0} + 3\diam{\cN_1} + r_n + 4\ep \\
& \leq 3\diam{\cN_0} + 3\diam{\cN_1} + 2.
\end{align*}
We verify that
\[
\mu_{\frac{1}{2}, n} (U_{2\ep}(\cN_n)) \geq 1 - 2\ep.
\]
Since $U_{2\ep}(M(\cZ_0\times\cZ_1)) \supset M(U_{2\ep}(\cZ_0) \times U_{2\ep}(\cZ_1))$, we have
\begin{align*}
& \mu_{\frac{1}{2}, n}(U_{2\ep}(\cN_n))  \\
\geq & M_* \pi_n(U_{2\ep}(M(\cZ_0\times\cZ_1)))  + \frac{1}{2}(\mu_{X_0}-{\pr_0}_* \pi_n)(U_{2\ep}(\cZ_0))+\frac{1}{2}(\mu_{X_1}-{\pr_1}_* \pi_n)(U_{2\ep}(\cZ_1)) \\
\geq & \pi_n(U_{2\ep}(\cZ_0) \times U_{2\ep}(\cZ_1))  + \frac{1}{2}(\mu_{X_0}-{\pr_0}_* \pi_n)(U_{2\ep}(\cZ_0))+\frac{1}{2}(\mu_{X_1}-{\pr_1}_* \pi_n)(U_{2\ep}(\cZ_1)) \\
\geq & \pi_n(Z_n \times Z_n) - {\pr_0}_*\pi_n(Z_n \setminus U_{2\ep}(\cZ_0)) - {\pr_1}_*\pi_n(Z_n \setminus U_{2\ep}(\cZ_1)) \\
& + \frac{1}{2}(\mu_{X_0}-{\pr_0}_* \pi_n)(U_{2\ep}(\cZ_0))+\frac{1}{2}(\mu_{X_1}-{\pr_1}_* \pi_n)(U_{2\ep}(\cZ_1)) \\
\geq & \mu_{X_0}(U_{\ep}(\cN_0)) + \mu_{X_1}(U_{\ep}(\cN_1)) -1 \geq 1 - 2\ep.
\end{align*}
Thus $\{X_{\frac{1}{2}, n}\}_{n=1}^\infty$ is precompact. There exists a $\square$-convergent subsequence of $\{X_{\frac{1}{2}, n}\}_{n=1}^\infty$ and its limit, denote by $X_{\frac{1}{2}}$, satisfies
\[
\square(X_i, X_{\frac{1}{2}}) \leq \frac{1}{2} \square(X_0, X_1), \quad i=0,1.
\]
The proof is completed.
\end{proof}

\begin{rem}
\begin{enumerate}
\item The Gromov-Prokhorov distance $\gp$ is also geodesic on $\X$.
\item Any geodesic metric space is locally path connected clearly. Hence $\X$ is locally path connected in the box topology.
\end{enumerate}
\end{rem}

On $(\X, \square)$, a geodesic between two distinct mm-spaces is never unique and it branches everywhere.

\begin{thm}\label{thm:multigeod}
For any two mm-spaces which are not mm-isomorphic, there exists a family of uncountably many pairwise-disjoint $\square$-geodesics between them.
Here, two geodesics are disjoint if they do not intersect anywhere except the endpoints.
In particular, every $\square$-geodesic branches everywhere.
\end{thm}

In order to construct a family of geodesics, we prepare the following lemma.

\begin{lem}\label{lem:geod_prod}
Let $\{X_t\}_{t \in [0,1]}$ be a $\square$-geodesic from an mm-space $X_0$ to an mm-space $X_1$ and let $r := \square(X_0, X_1)$.
Take any $r$-Lipschitz function $f \colon [0,1] \to [0,+\infty)$ with $f(0) = f(1) = 0$ and any mm-space $Z$ with $0 < \diam{Z} \le 1$ and define an mm-space $Y_t$, $t \in [0,1]$, by
\[
Y_t := X_t \times_\infty f(t)Z.
\]
Then $\{Y_t\}_{t \in [0,1]}$ is also $\square$-geodesic from $X_0$ to $X_1$.
\end{lem}

\begin{proof}
Take any $s, t \in [0, 1]$ and fix them.
It is sufficient to prove that
\[
\square(Y_s, Y_t) \le |s-t|r
\]
by the triangle inequality.
Since $\square(X_s, X_t) = |s-t|r$, there exist a coupling $\pi \in \Pi(\mu_{X_s},\mu_{X_t})$ and a closed set $S \subset X_s\times X_t$ such that
\[
\max\{\dis{S}, 1 - \pi(S)\} = |s-t|r
\]
by Lemma \ref{box_opt}.
Here we define a coupling $\pi'\in \Pi(\mu_{X_s}\otimes \mu_Z,\mu_{X_t} \otimes \mu_Z)$ and a closed subset $S' \subset Y_s\times Y_t$ by
\[
\pi' := (\pr_1, \pr_3, \pr_2, \pr_3)_* (\pi \otimes \mu_Z), \quad
S' := \left\{(x, z, y, z) \midd (x, y) \in S, z \in Z\right\}.
\]
Then we see that
\[
\pi'(S') = (\pi \otimes \mu_Z)(S \times Z) = \pi(S) \ge 1 - |s - t| r.
\]
Moreover, for any $(x,z,y,z),(x',z',y',z') \in S'$, we have
\begin{align*}
&\left|\max\{d_{X_s}(x,x'), f(s)d_Z(z,z')\} - \max\{d_{X_t}(y,y'), f(t)d_Z(z,z')\}\right| \\
&\le \max\{|d_{X_s}(x,x')-d_{X_t}(y,y')|, |f(s)-f(t)|d_Z(z,z')\}\\
&\le \max\{\dis{S}, |f(s)-f(t)|\} \\
&\le |s-t|r,
\end{align*}
which implies that $\dis{S'} \le |s-t|r$. Therefore we obtain
\[
\square(Y_s, Y_t) \le \max\{\dis{S'}, 1-\pi'(S')\} \le |s-t|r
\]
by Lemma \ref{box_opt} and then $\{Y_t\}_{t\in[0,1]}$ is a $\square$-geodesic. The proof is completed.
\end{proof}

\begin{proof}[Proof of Theorem \ref{thm:multigeod}]
Let $X_0$ and $X_1$ be mm-spaces and assume $r := \square(X_0, X_1) > 0$.
By Theorem \ref{box_geod}, there exists a $\square$-geodesic $\{X_t\}_{t \in [0,1]}$ from $X_0$ to $X_1$.
We take a function $f$ and an mm-space $Z$ satisfying the assumption of Lemma \ref{lem:geod_prod}, e.g.,
\[
f(t) := r \min\{t, 1-t\} \quad \text{ and } \quad Z := ([0,1], |\cdot|, \mathcal{L}^1).
\]
For any $s,t \in [0,1]$, an mm-space $Y_{s,t}$ is defined by
\[
Y_{s, t} := X_t \times_\infty sf(t) Z.
\]
By Lemma \ref{lem:geod_prod}, $\{Y_{s,t}\}_{t \in [0,1]}$ is a geodesic from $X_0$ to $X_1$ for every $s$.
We prove that $\{Y_{s,t}\}_{t \in [0,1]}$ and $\{Y_{s',t}\}_{t \in [0,1]}$ are disjoint if $s \neq s'$.
It is sufficient to prove that $Y_{s, t}$ and $Y_{s',t}$ are not mm-isomorphic for any $s,s',t \in [0,1]$ with $s<s'$.
The map $\phi := \id_{X_t \times Z}$ is a 1-Lipschitz measure-preserving map from $Y_{s',t}$ to $Y_{s,t}$.
If $Y_{s',t}$ and $Y_{s,t}$ are mm-isomorphic, then the map $\phi$ must be an isometry from $Y_{s',t}$ to $Y_{s,t}$ in the same way as \cite{MMG}*{Proof of Lemma 2.12}, which is a contradiction.
Thus $Y_{s',t}$ and $Y_{s,t}$ are not mm-isomorphic to each other.
Therefore we obtain a family $\{t\mapsto Y_{s,t}\}_{s \in [0,1]}$ of uncountably many pairwise-disjoint geodesics . The proof is completed.
\end{proof}

\begin{cor}\label{cor:geod_noncpt}
For any two mm-spaces $X_0$ and $X_1$ which are not mm-isomorphic and for any $t \in [0,1]$, the set
\[
[X_0, X_1]_t := \left\{Z \in \X \midd \square(X_0, Z) = t\square(X_0,X_1)\text{ and } \square(X_1, Z) = (1-t)\square(X_0,X_1)\right\}
\]
is not compact with respect to the box topology.
\end{cor}

\begin{proof}
Take a $\square$-geodesic $\{X_t\}_{t\in[0,1]}$ from $X_0$ to $X_1$ and a function $f$ satisfying the assumption of Lemma \ref{lem:geod_prod}.
Then the mm-space
\[
Z_n := X_t \times_\infty f(t)([0,1]^n, \|\cdot\|_\infty, \mathcal{L}^n), \quad n=1,2,\ldots
\]
is in the set $[X_0, X_1]_t$ for any $n$.
However, this sequence $\{Z_n\}_{n=1}^\infty$ have no $\square$-convergent subsequence (see \cite{MMG}*{Proposition 7.37}).
This completes the proof.
\end{proof}

\begin{rem}
\begin{enumerate}
\item Theorem \ref{thm:multigeod} shows that the Alexandrov curvature of $\X$
is not bounded from below nor from above with respect to the box metric $\square$.
\item In the Gromov-Hausdorff space case, one can see the analogous statements of Lemma \ref{lem:geod_prod} in \cite{I}*{Proposition 5.3} and of Corollary \ref{cor:geod_noncpt} in \cite{Bori}.
The construction of a family of geodesics on the Gromov-Hausdorff space has studied in \cites{I, MW}.
\end{enumerate}
\end{rem}

\section{$\X$ and $\Pi$ are locally path connected}\label{sec:loc_path_conn}

The goal of this section is to prove Theorem \ref{loc_path_conn}. For the concentration and weak topologies, it is difficult to obtain a geodesic at present, but it is possible to construct a continuous path in a small ball. We prepare some lemmas to prove Theorem \ref{loc_path_conn}.

\begin{prop}\label{monotone_path}
Let $X_0$ and $X_1$ be two mm-spaces with $X_0 \prec X_1$ and let $f\colon X_1 \to X_0$ be a $1$-Lipschitz measure-preserving map. For any $0 < t < 1$, we define a metric $d_{X_t}$ on $X_1$ by
\[
d_{X_t}(x, x') := (1-t)d_{X_0}(f(x), f(x')) + td_{X_1}(x, x'), \quad x,x' \in X_1,
\]
and define an mm-space
\[
X_t := (X_1, d_{X_t}, \mu_{X_1}).
\]
Then the map $[0, 1] \ni t \mapsto X_t$ is a $\square$-continuous path from $X_0$ to $X_1$ and is monotone with respect to the Lipschitz order, that is, $X_s \prec X_t$ for every $0 \leq s \leq t \leq 1$.
\end{prop}

\begin{proof}
For any $s, t \in (0,1]$ with $s \leq t$, since $d_{X_s} \leq d_{X_t}$, we have $X_s \prec X_t$.
Moreover, for any $t \in (0,1]$, the map $f$ is also 1-Lipschitz with respect to $d_{X_0}$ and $d_{X_t}$, which implies $X_0 \prec X_t$.
Thus we obtain the monotonicity of $t \mapsto X_t$.

We next prove the (uniform) continuity with respect to $\square$. We take a real number $\ep > 0$. By the inner regularity of $\mu_{X_1}$, there exists a compact subset $K \subset X_1$ such that
\[
\mu_{X_1}(K) \geq 1-\ep.
\]
If two real numbers $s, t \in [0,1]$ satisfy $|s - t| \leq (\diam{K})^{-1}\ep$, then $\square(X_s, X_t) \leq \ep$ holds.
Indeed, in the case of $s, t \in (0,1]$, letting $\pi := (\id_{X_1}, \id_{X_1})_* \mu_{X_1}$ and $S := \left\{(x, x) \midd x \in K \right\}$, we have $\pi(S) = \mu_{X_1}(K) \geq 1-\ep$ and
\begin{align*}
\dis{S} & = \sup_{x, x' \in K} |d_{X_s}(x,x') - d_{X_t}(x,x')| \\
& = |s - t| \sup_{x, x' \in K} |d_{X_1}(x,x') - d_{X_0}(f(x),f(x'))| \\
& \leq |s - t| \diam{K} \leq \ep.
\end{align*}
Thus we obtain $\square(X_s, X_t) \leq \ep$ by Lemma \ref{box_opt}.
Similarly, if $s = 0$, then we just put $\pi := (f, \id_{X_1})_* \mu_{X_1}$ and $S := \left\{(f(x), x) \midd x \in K \right\}$.
The proof is completed.
\end{proof}

\begin{prop}\label{pyramid_path}
Let $\cP$ be a pyramid and let $\ep > 0$. There exist an mm-space $X \in \cP$ and a $\rho$-continuous path $\gamma \colon [0,1] \to \Pi$ joining $\cP_X$ and $\cP$ such that
\[
\rho(\gamma(t), \cP) < \ep \quad \text{ and } \quad \gamma(t) \subset \cP
\]
for all $t \in [0,1]$.
\end{prop}

\begin{proof}
Let $\{Y_m\}_{m=1}^\infty$ be an approximation of $\cP$. By Proposition \ref{monotone_path}, for each $m$, there exists a $\square$-continuous path $\gamma_m \colon [0,1] \to \X$ from $Y_m$ to $Y_{m+1}$ with $Y_m \prec \gamma_m(t) \prec Y_{m+1}$ for all $t$.
We define a map $\gamma \colon [0, 1) \to \Pi$ by
\[
\gamma(t) := \cP_{\gamma_m(2 - 2^m(1-t))} \quad \text{ if } 1-2^{-m+1} \leq t \leq 1 -2^{-m}.
\]
Since $Y_m$ converges weakly to $\cP$ as $m\to\infty$, $\gamma(t)$ converges weakly to $\cP$ as $t \to 1$. Thus $\gamma$ is a $\rho$-continuous path from $\cP_{Y_1}$ and $\cP$ with $\gamma(t) \subset \cP$ for all $t$. Cutting $\gamma$ if it is required, we obtain the desired one.
\end{proof}

\begin{proof}[Proof of Theorem \ref{loc_path_conn}]
We first prove that $\X$ with the concentration topology is locally path connected.
If not, there exist an mm-space $X$, a real number $\ep > 0$, and a sequence $\{Y_n\}_{n=1}^\infty$ of mm-spaces concentrating to $X$ such that for any $\conc$-continuous path $\gamma \colon [0,1] \to \X$ from $Y_n$ to $X$, there exists $t \in [0,1]$ such that
\[
\conc(\gamma(t), X) > \ep.
\]
By Proposition \ref{mmg6.2}, there exists a sequence $\{Z_n\}_{n = 1}^\infty$ of mm-spaces $\square$-converging to $X$ with $Z_n \prec Y_n$ for every $n$. By Proposition \ref{monotone_path} and Theorem \ref{box_geod}, for each $n$, there exist $\square$-continuous paths $\gamma_n^1$ and $\gamma_n^2$ such that
\begin{itemize}
\item $\gamma_n^1 \colon [0,1] \to \X$ joins $Y_n$ to $Z_n$ and satisfies $Z_n \prec \gamma_n^1(t) \prec Y_n$ for any $0\leq t \leq 1$,
\item $\gamma_n^2 \colon [0,1] \to \X$ is a $\square$-geodesic from $Z_n$ to $X$.
\end{itemize}
Joining the two paths $\gamma_n^1$ and $\gamma_n^2$, we obtain a continuous path from $Y_n$ to $X$. By the assumption and by
\[
\limsup_{n\to \infty} \sup_{t \in [0,1]} \conc(\gamma_n^2(t), X) \leq \lim_{n\to \infty} \square(Z_n, X) = 0,
\]
there exists $\{t_n\}_{n=1}^\infty \subset [0,1]$ such that
\[
\conc(\gamma_n^1(t_n), X) > \ep
\]
for any sufficiently large $n$. On the other hand, $\{\gamma_n^1(t_n)\}_{n=1}^\infty$ must concentrate to $X$ as $n \to \infty$. Indeed, since $Z_n \prec \gamma_n^1(t_n) \prec Y_n$ for all $n$, a limit of a weak convergent subsequence of $\{\gamma_n^1(t_n)\}_{n =1}^\infty$ must be $X$. This is a contradiction. Therefore $\X$ is locally path connected in the concentration topology.

We next prove that $\Pi$ is locally path connected. The outline of the proof is same as that in the first half. Suppose that $\Pi$ is not locally path connected. There exist a pyramid $\cP$, a real number $\ep > 0$, and a sequence $\{\cP_n\}_{n=1}^\infty$ of pyramids converging weakly to $\cP$ such that for any $\rho$-continuous path $\gamma \colon [0,1] \to \Pi$ from $\cP_n$ to $\cP$, there exists $t \in [0,1]$ such that
\[
\rho(\gamma(t), \cP) > \ep.
\]
For every $m$ and $n$, by Proposition \ref{pyramid_path}, there exist mm-spaces $X_m \in \cP$, $Y_n \in \cP_n$ and $\rho$-continuous paths $\gamma_n^1$ and $\gamma_m^4$ such that
\begin{itemize}
\item $\gamma_n^1 \colon [0,1] \to \Pi$ joins $\cP_n$ to $\cP_{Y_n}$ and satisfies $\rho(\gamma_n^1(t), \cP_n) \leq n^{-1}$ for any $0\leq t \leq 1$.
\item $\gamma_m^4 \colon [0,1] \to \Pi$ joins $\cP_{X_m}$ to $\cP$ and satisfies
$\rho(\gamma_m^4(t), \cP) \leq m^{-1}$ for any $0 \leq t \leq 1$.
\end{itemize}
In particular, $\{\cP_{Y_n}\}_{n=1}^\infty$ converges weakly to $\cP$.
Thus, by the definition of the weak convergence, for any $m$, there exists a sequence $\{Z_{mn}\}_{n=1}^\infty$ of mm-spaces such that
\[
\lim_{n\to\infty} \square(Z_{mn}, X_m) = 0 \quad \text{ and } \quad  Z_{mn} \prec Y_n \text{ for every } n.
\]
For every $m$, we choose $n=n(m)$ as
\[
\square(Z_{mn(m)}, X_m) \le m^{-1} \quad \text{ and } \quad  \lim_{m\to\infty} n(m) = \infty,
\]
and put
\[
Y_m := Y_{n(m)}, \quad Z_m := Z_{mn(m)}, \quad \text{ and } \quad \gamma_m^1 := \gamma_{n(m)}^1.
\]
By Proposition \ref{monotone_path} and Theorem \ref{box_geod}, for each $m$, there exist $\square$-continuous paths $\gamma_m^2$ and $\gamma_m^3$ such that
\begin{itemize}
\item $\gamma_m^2 \colon [0,1] \to \X$ joins $Y_m$ to $Z_m$ and satisfies $Z_m \prec \gamma_m^2(t) \prec Y_m$ for any $0\leq t \leq 1$.
\item $\gamma_m^3 \colon [0,1] \to \X$  is a $\square$-geodesic from $Z_m$ to $X_m$.
\end{itemize}
Joining the four paths $\gamma_m^1$, $\gamma_m^2$, $\gamma_m^3$, and $\gamma_m^4$, we obtain a $\rho$-continuous path from $\cP_m = \cP_{n(m)}$ to $\cP$. By the assumption and by
\[
\lim_{m\to \infty} \sup_{t \in [0,1]} \rho(\gamma_m^1(t), \cP) = \lim_{m\to \infty} \sup_{t \in [0,1]} \rho(\gamma_m^3(t), \cP) = \lim_{m\to \infty} \sup_{t \in [0,1]} \rho(\gamma_m^4(t), \cP) = 0,
\]
there exists $\{t_m\}_{m=1}^\infty \subset [0,1]$ such that
\[
\rho(\gamma_m^2(t_m), \cP) > \ep
\]
for any sufficiently large $m$. On the other hand, $\{\gamma_m^2(t_m)\}_{m=1}^\infty$ must converge weakly to $\cP$ as $m \to \infty$. This is a contradiction. Therefore $\Pi$ is locally path connected. The proof of the theorem is completed.
\end{proof}

Theorem \ref{contra} and Theorem \ref{loc_path_conn} together imply Corollary \ref{Peano} directly. We recall the {\it Peano space} through the following theorem.

\begin{thm}[Hahn-Mazurkiewicz theorem]
A Hausdorff space is compact, connected, metrizable, and locally connected if and only if it is a continuous image of the unit closed interval $[0, 1]$.
{\rm (}Such a space is called a Peano space.{\rm )}
\end{thm}

\section{Revisit the weak topology on $\Pi$}\label{sec:hypersp}
Let $X$ be a Hausdorff space and let $\F(X)$ be the set of all closed subsets of $X$.

We recall that for a given net $\{A_\lm\}_{\lm \in \Lm}$ in $\F(X)$, the {\it upper closed limit} and the {\it lower closed limit} of $\{A_\lm\}_{\lm \in \Lm}$ are defined as
\begin{align*}
\Ls{A_\lm} & := \left\{x \in X \midd U_x \cap A_\lm \neq \emptyset \text{ cofinally for every neighborhood } U_x \text{ of } x \right\}, \\
\Li{A_\lm} & := \left\{x \in X \midd U_x \cap A_\lm \neq \emptyset \text{ residually for every neighborhood } U_x \text{ of } x \right\}.
\end{align*}
Note that $\Ls{A_\lm}$ and $\Li{A_\lm}$ are closed subsets of $X$ and $\Li{A_\lm} \subset \Ls{A_\lm}$. A element of $\Ls{A_\lm}$ is called a {\it cluster point} of $\{A_\lm\}_{\lm \in \Lm}$ and a element of $\Li{A_\lm}$ a {\it limit point}. In the case that $X$ is a metric space with metric $d$, we see that
\[
\Ls{A_\lm} = \left\{x \in X \midd {\textstyle \liminf_{\lm} d(x, A_\lm)} = 0 \right\} \text{ and }
\Li{A_\lm} = \left\{x \in X \midd {\textstyle\limsup_{\lm} d(x, A_\lm)} = 0 \right\},
\]
where $\liminf_{\lm} a_\lm := \sup_{\lm \in \Lm} \inf_{\lm' \geq \lm} a_\lm$ and $\limsup_{\lm} a_\lm := \inf_{\lm \in \Lm} \sup_{\lm' \geq \lm} a_\lm$ for a net $\{a_\lm\}_{\lm\in\Lm}$ of real numbers.

\begin{dfn}\label{KPconv}
A net $\{A_\lm\}_{\lm \in \Lm}$ {\it Kuratowski-Painlev\'e converges} to $A \in \F(X)$ provided that
\[
\Li{A_\lm} = \Ls{A_\lm} = A.
\]
\end{dfn}

It is well-known that the Kuratowski-Painlev\'e convergence is topological (that is, there exists a topology achieving the convergence) if and only if $X$ is locally compact. On the other hand, the finest topology whose convergence is weaker than the Kuratowski-Painlev\'e convergence always exists. This topology, write $\tK$, is called the {\it topologization} of the Kuratowski-Painlev\'e convergence, or {\it convergence topology} historically.

We next recall the {\it Fell topology} $\tF$ on $\F(X)$. The Fell topology is deeply related to the Kuratowski-Painlev\'e convergence. This topology is determined by the following subbase:
\[
\left\{\{A \in \F(X) : A \cap V \neq \emptyset\} \midd V \text{ is open} \right\}
\cup \left\{\{A \in \F(X) : A \subset X \setminus K \} \midd K \text{ is compact} \right\}.
\]
The Kuratowski-Painlev\'e convergence implies $\tF$-convergence, so that $\tF \subset \tK$. If $X$ is locally compact, then the $\tF$-convergence implies the Kuratowski-Painlev\'e convergence conversely. In this case, the Fell topology achieves the Kuratowski-Painlev\'e convergence.

The known results used in this paper are listed as follows.

\begin{thm}[Mrowka's theorem, cf.~\cite{B}*{Theorems 5.2.11 and 5.2.12}]
Let $X$ be a Hausdorff space. Any net $\{A_\lm\}_{\lm\in\Lm}$ in $\F(X)$ has a Kuratowski-Painlev\'e convergent subnet. Moreover, if $X$ is second countable, then any sequence $\{A_n\}_{n=1}^\infty$ in $\F(X)$ has a Kuratowski-Painlev\'e convergent subsequence.
\end{thm}

We remark that a subnet of a sequence is not necessarily a subsequence. Mrowka's theorem says the compactness and the sequentially compactness of both $\tF$ and $\tK$.

\begin{thm}[cf.~\cite{B}*{Proposition 5.1.2}, \cite{F}*{4A2T}]
Let $X$ be a Hausdorff space. Then the following {\rm (1)} and {\rm (2)} hold.
\begin{enumerate}
\item Both $\tF$ and $\tK$ are $T_1$ {\rm (}Fr\'echet{\rm )}.
\item $\tF$ is Hausdorff if and only if $X$ is locally compact.
\end{enumerate}
\end{thm}

\begin{thm}[\cite{BRL}*{Theorems 3.12 and 3.13}]
Let $X$ be a metrizable space. Then the following {\rm (1)--(3)} hold.
\begin{enumerate}
\item $\tK$ is sequential if and only if $X$ is separable.
\item $\tK = \tF$ if and only if $X$ has at most one point that has no compact neighborhood.
\item $\tF$ is sequential if and only if $X$ is separable and $X$ has at most one point that has no compact neighborhood.
\end{enumerate}
\end{thm}

\begin{thm}[cf.~\cite{B}*{Theorem 5.2.10}]
Let $X$ be a first countable Hausdorff space and let $A$ and $A_n$, $n=1,2,\ldots$, be a closed subsets of $X$. Then $\{A_n\}_{n=1}^\infty$ $\tF$-converges to $A$ if and only if $\{A_n\}_{n=1}^\infty$ Kuratowski-Painlev\'e converges to $A$. In particular, a limit of any $\tF$-convergent sequence is unique.
\end{thm}

We now consider the space $\F(\X, \square)$. Theorem \ref{no_cpt_nbd} and some properties of $(\X, \square)$ imply that
\begin{itemize}
\item the Kuratowski-Painlev\'e convergence on $\F(\X, \square)$ is not topological,
\item $(\F(\X, \square), \tF)$ is compact and $T_1$, but neither Hausdorff nor sequential,
\item $(\F(\X, \square), \tK)$ is compact, $T_1$, and sequential,
\item $\tF \subsetneq \tK$,
\item $\tF$ and $\tK$ are sequentially equivalent on $\F(\X, \square)$.
\end{itemize}

\begin{rem}
\begin{enumerate}
\item $\tK$ is the sequentially modification of the Fell topology $\tF$. The sequentially modification of a topology is a stronger topology whose open sets consist of all sequentially open sets of the original topology.
\item Viewing $\tK$ from another angle, if $X$ is a metrizable space, then $\tK$ is the infimum of the Wijsman topologies $\tau_{\mathrm{W}_d}$, where $d$ runs over all compatible metrics on $X$. For a metric space $(X, d)$, the Wijsman topology $\tau_{\mathrm{W}_d}$ on $\F(X)$ is the weakest topology such that the function $A \mapsto d(x, A)$ on $\F(X)$ is continuous for every $x \in X$.
Here, if there exists a compatible metric $d$ on $X$ such that $\tK = \tau_{\mathrm{W}_d}$, then $X$ must be locally compact (see \cite{B85}). Over $(\X,\square)$, there is no minimum of the Wijsman topologies.
\item The authors do not know whether $(\F(\X, \square), \tK)$ is Hausdorff or not.
\end{enumerate}
\end{rem}

We now prove Theorem \ref{emb}.

\begin{proof}[Proof of Theorem \ref{emb}]
Since both $\Pi $ and $(\F(\X, \square), \tK)$ are sequential, it is sufficient to prove that, for any sequence $\{\cP_n\}_{n=1}^\infty$ of pyramids, the weak convergence and the $\tK$-convergence coincide with each other. Actually, these coincide with the Kuratowski-Painlev\'e convergence respectively. This completes the proof.
\end{proof}

The final topic in this paper starts with the following fact.

\begin{prop}[\cite{MMG}*{Corollary 6.15}]
Any pyramid is $\conc$-closed.
\end{prop}

This means that $\Pi \subset \F(\X, \conc) \subset \F(\X, \square)$. The following observation is interesting.

\begin{prop}\label{box-conc}
The Kuratowski-Painlev\'e convergence on $\F(\X, \square)$ and $\F(\X, \conc)$ coincide for any sequences of pyramids.
\end{prop}

\begin{proof}
Let $\{\cP_n\}_{n=1}^\infty$ be a sequence of pyramids. We denote by $\Ls_{\square} \cP_n$ and $\Li_{\square} \cP_n$ the upper closed limit and the lower closed limit with respect to the box distance function $\square$. Similarly, we denote by $\Ls_{\mathrm{conc}} \cP_n$ and $\Li_{\mathrm{conc}} \cP_n$ them with respect to the observable distance function $\conc$.
Since $\conc \leq \square$,
\[
\Ls_{\square} \cP_n \subset \Ls_{\mathrm{conc}} \cP_n \quad \text{ and } \quad \Li_{\square} \cP_n \subset \Li_{\mathrm{conc}} \cP_n
\]
are trivial. We prove $\Ls_{\mathrm{conc}} \cP_n \subset \Ls_{\square} \cP_n$. Let $X \in \Ls_{\mathrm{conc}} \cP_n$. Then there exist a subsequence $\{n_i\}_{i=1}^\infty$ of $\{n\}$ and mm-spaces $X_i \in \cP_{n_i}$, $i=1,2,\ldots$ such that
$X_i$ concentrates to $X$ as $i \to \infty$. By Proposition \ref{mmg6.2}, we find a sequence $\{Y_i\}_{i=1}^\infty$ of mm-spaces $\square$-converging to $X$ with $Y_i \prec X_i$ for every $i$. Since $\cP_{n_i}$ is a pyramid, we have $Y_i \in \cP_{n_i}$ and hence
\[
\liminf_{n\to\infty} \square(X, \cP_n) \leq \liminf_{i\to\infty} \square(X, \cP_{n_i}) \leq \lim_{i\to\infty} \square(X, Y_i) = 0,
\]
which implies that $X \in \Ls_{\square} \cP_n$. We obtain $\Ls_{\mathrm{conc}} \cP_n \subset \Ls_{\square} \cP_n$. Similarly, $\Li_{\mathrm{conc}} \cP_n \subset \Li_{\square} \cP_n$ is obtained. The proof is completed.
\end{proof}

Proposition \ref{box-conc} leads to another embedding as follows.

\begin{thm}\label{emb_conc}
The inclusion map $\Pi \ni \cP \mapsto \cP \in (\F(\X, \conc), \tK)$ is a topological embedding map.
\end{thm}

Note that $(\F(\X, \conc), \tK)$ is also compact, $T_1$, and sequential. This $\tK$ is just the topologization of the Kuratowski-Painlev\'e convergence over $(\X, \conc)$ which is unrelated to $(\F(\X, \square), \tK)$ by definition.

\section{Further questions}
The following question remains.

\begin{qst}
Is $\X$ with the concentration topology $\sigma$-compact? Equivalently, is  $\Pi\setminus\X$ a $G_\delta$ subset of $\Pi$?
\end{qst}
If this question is true, then $\Pi\setminus \X$ is a Baire space.

\begin{qst}
Are $(\X, \conc)$ and $(\Pi,\rho)$ geodesic spaces?
\end{qst}

\begin{qst}
Is $(\F(\X, \square), \tK)$ Hausdorff (metrizable)?
\end{qst}
\begin{qst}
Is there a relation between $(\F(\X, \square), \tK)$ and $(\F(\X, \conc), \tK)$?
\end{qst}
\begin{qst}
Can $\Pi$ be embedded into $(\F(\X, \square), \tF)$ or $(\F(\X, \square), \tau_{\mathrm{W}_\square})$ topologically?
\end{qst}

\subsection*{Acknowledgements}
The authors thank Professors Takumi Yokota and Yoshito Ishiki for their valuable comments.

\end{document}